\documentclass[12pt, 14paper, reqno]{amsart}
\usepackage{amsmath,amsfonts,amssymb}

\theoremstyle{plain}
\newtheorem{theorem}{Theorem}
\newtheorem{lemma}{Lemma}

\theoremstyle{proof}
\theoremstyle{definition}

\theoremstyle{remark}

\numberwithin{equation}{section}
\numberwithin{lemma}{section}
\numberwithin{theorem}{section}

\usepackage{amsmath}
\usepackage{amsfonts}   
\usepackage{amssymb}
\usepackage{amssymb, amsmath, amsthm}
\usepackage[breaklinks]{hyperref}

\theoremstyle{thmrm}

\author{S. Banerjee}
\address{Soumyarup Banerjee,  Harish-Chandra Research Institute, HBNI,
Chhatnag Road, Jhunsi,  Allahabad 211 019,  India}
\email{ soumyabanerjee@hri.res.in}

\title[Signs Of Fourier Coefficients Of two Cusp Forms]{A Note On Signs Of Fourier Coefficients Of two Cusp Forms}


\begin{document}

\maketitle

\begin{abstract}
Kohnen and Sengupta proved that two cusp forms of different integral weights with real algebraic Fourier coefficients have infinitely many Fourier coefficients of the same as well as of opposite sign, up to the action of a Galois automorphism. Recently Gun, Kohnen and Rath strengthen their result by comparing the simultaneous sign changes of Fourier coefficients of two cusp forms with arbitrary real Fourier coefficients. The simultaneous sign changes of Fourier coefficients of two same integral weight cusp forms follow from an earlier work of Ram Murty. In this note we compare the signs of the Fourier coefficients of two cusp forms simultaneously for the congruence subgroup $\Gamma_0(\mathit{N})$ where the coefficients lie in an arithmetic progression. Next we consider an analogous question for the particular sparse sequences of Fourier coefficients of normalized Hecke eigen cusp forms for the full modular group.
\end{abstract}

\section{Introduction}
Throughout the paper, let $z\in \mathfrak{H}$ be an element of the Poincar\'e upper-half plane and $q=e^{2\pi iz}$. Let $S_k$ denote the space of cusp forms of weight k for full modular group $SL_2(\mathbb{Z})$ and $S_k(N)$ denote the space of cusp forms of weight $k$ for the congruence subgroup $\Gamma_0(N)$.

Sign changes of Fourier coefficients of cusp forms in one or in several variables have been studied in various aspects by many authors. It is known that, if the Fourier coefficients of a cusp form are real then they change signs infinitely often \cite{Koh-1}.  Further, many quantitative results for the number of sign changes for the sequence of the Fourier coefficients have been established. The sign changes of the subsequence of the Fourier coefficients at prime numbers was first studied by M. R.  Murty \cite{Ram}. Later, J. Meher et. al.  \cite{Jab} investigated that for a normalized Hecke eigen cusp form with Fourier coefficients $\{a(n)\}$, the subsequences $\{a(n^{j})\}_{n\geq1} ~(j=2,3,4)$ of the Fourier coefficients change signs infinitely often. W.Kohnen and Y.Martin \cite{Koh-2}  generalized their result by showing the subsequence $\{a(p^{jn})\}_{n\geq0}$  have infinitely many sign changes for almost all primes $p$ and $ j\in\mathbb{N}$ .

Recently in 2009, the question of simultaneous sign change of Fourier coefficients of two cusp forms of different integral weights with real algebraic Fourier coefficients was considered by W. Kohnen and J. Sengupta. They showed that for $k_1 \neq k_2$ if $f$ and $g$ is in $S_{k_i}(\Gamma_0(\mathbb{N}))$ for $i=1,2$ with totally real algebraic Fourier coefficients $\{a(n)\}$ and $\{b(n)\}$ respectively for $n\geq 1$ with $a(1)=1=b(1)$, then there exist an element $\sigma$ of the absolute Galois group $Gal(\mathbb{\overline{Q}}/\mathbb{Q})$ such that $a(n)^{\sigma}b(n)^{\sigma}<0$ for infinitely many $n$. Later Gun, Kohnen and Rath extended their result by showing that if $f$ and $g$ have real Fourier coefficients $\{a(n)\}$ and $\{b(n)\}$ respectively for $n\geq 1$ and $a(1)b(1)\neq 0$, then there exist infinitely many $n$ such that $a(n)b(n)>0$ and infinitely many $n$ such that $a(n)b(n)<0$. When $k_1=k_2$, the simultaneous sign changes of Fourier coefficients of $f$ and $g$ follows from an earlier work of Ram Murty \cite{Ram}.

In this article, we have considered two normalized Hecke eigen cusp forms $f$ and $g$ of weight $1<k_2<k_1$ for $\Gamma_0(\mathbb{N})$ with Fourier coefficients $\{a(n)\}$ and $\{b(n)\}$ respectively and have studied the simultaneous sign changes of Fourier coefficients where the coefficients are in arithmetic progression. We have also studied the simultaneous sign changes of the sparse sequences $\{a(n^j)\}$ and $\{b(n^j)\}$ for $j= 2, 3, 4$ respectively of $f$ and $g$ over full modular group $SL_2(\mathbb{Z})$. 
%
%


\section{Statement of the results}
The following two theorems are the main results alluded off earlier.

\begin{theorem}\label{thm2}
Let $f, g \in S_{k_i} (N)$ (with $i = 1,2$ respectively) having Fourier coefficients $a(n)$ and $b(n)$ respectively which are normalized, i.e. $a(1) = b(1) =1$. Then for any $m\in \mathbb{N}$ and $l \in \mathbb{Z}$ with $(l,m)=1$, the subsequences $(a(n))_{n\equiv l \pmod m}$ and $(b(n))_{n\equiv l \pmod m}$ have infinitely many terms of same as well as of different signs simultaneously.  
\end{theorem}

\begin{theorem}\label{thm3}
Let $f,g \in S_{k_i}$ (with $i = 1,2$ respectively) be normalized Hecke eigenform with Fourier coefficients $a(n)$ and $b(n)$ respectively. Then for $ j= 2, 3, 4 $ the sparse sequences $\{a(n^j)\}_{n\geq1}$ and $\{b(n^j)\}_{n\geq1}$ have infinitely many terms of same as well as of different signs simultaneously.
\end{theorem}

\section{Proof of Theorem \ref{thm2}}
We begin by stating and sometimes outlining a proof of a  series of lemmas. Let us define the counting function
\begin{equation}\label{eq1}
I_l(n):= 
\begin{cases}
&1 ~~~~~\text{~~~~~~if $n\equiv l \pmod m$} \\
& 0~~~~~ \text{~~~~~~otherwise.}
\end{cases}
\end{equation}
with $m$ and $l$ as before. 
\begin{lemma}\label{lem1}
Let $\chi$ be any Dirichlet character modulo N. Consider $f\in S_k(N,\chi)$ having Fourier coefficients $a(n)\in \mathbb{R}$ and  
$l, m $ be co-prime positive integers. If 
$ g(z) := \sum\limits_{n=1}^{\infty}I_l(n)a(n)q^n $, 
then $g\in S_k(\Gamma_1(Nm^2))$.
\end{lemma}
\begin{proof}
Let $\{\psi_r: 1\leq r\leq \phi(m)\}$ be the group of $\phi(m)$ Dirichlet characters modulo $m$ where $\phi$ is the Euler phi function. Now applying the orthogonality relation of characters we have,
\begin{eqnarray}
g(z)
&=&\sum_{n=1}^{\infty}\bigg(\frac{1}{\phi(m)} \sum_{r=1}^{\phi(m)}\psi_r(n)\overline{\psi_r}(l)\bigg)a(n)q^n \nonumber\\
&=&\sum_{r=1}^{\phi(m)}\frac{\overline{\psi_r}(l)}{\phi(m)}
\sum_{n=1}^\infty \psi_r(n)a(n)q^n\nonumber\\
&=&\sum_{r=1}^{\phi(m)}\alpha_{\psi_r}f_{\psi_r}
\end{eqnarray}
where $\alpha_{\psi_r}=\frac{\overline{\psi_r}(l)}{\phi(m)}$ and $f_{\psi_r}=\sum\limits_{n=1}^\infty \psi_r(n)a(n)q^n$. Now 
$f_{\psi_r} \in S_k(Nm^2,\chi\psi_{r}^2) \subset S_k(\Gamma_1(Nm^2))$ for all $r$ with $1\leq r \leq \phi(m)$. ( cf. \cite{Kob}). Hence $g \in S_k(\Gamma_1(Nm^2))$.
\end{proof}

\begin{lemma}\label{lem2}
Suppose $f\in S_k(\Gamma_1(N))$. Then its associated $L$ function $L(s,f)$ has an analytic continuation to the full s-plane.
\end{lemma}
\begin{proof}
We refer [\cite{Dia}, Theorem 5.10.2] for the proof.
\end{proof}

\begin{lemma}[Landau \cite{Lan}]\label{lem4}
Suppose that $d(n)\geq 0$  for all but finitely many $n$'s and that the Dirichlet series 
$$ 
\Psi(s) = \sum_{n\geq 1} \frac{d(n)}{n^{s}} 
$$ 
has finite abscissa of convergence $\sigma_c$. Then $ \Psi(s)$ has a singularity on the real axis at the point $s=\sigma_c$ .
\end{lemma} 
%

\subsection{Proof of Theorem \ref{thm2}}

Consider,
$$
f_{1}(z)=\sum_{n\geq 1}I_l(n)a(n)q^{n}
$$
and   
$$
g_{1}(z)=\sum_{n\geq 1}I_l(n)b(n)q^{n},
$$ 
with $I_l(n)$ is as in \eqref{eq1}.
Applying Lemma \ref{lem1} we have $f_1\in S_{k_1}(\Gamma_1(Nm^2))$ and $g_1 \in S_{k_2}(\Gamma_1(Nm^2))$. 
Let 
\begin{equation}
R_{f_{1},g_{1}}(s) = \sum_{n\geq 1}\dfrac{I_l(n)a(n)b(n)}{n^s}
\end{equation} 
be the Rankin-Selberg Dirichlet series attached to $ f_1 $ and $g_1$.
Let us take $k_1>k_2$ without any loss of generality and set, 
\begin{equation}\label{eq6}
R^*_{f_1,g_1}(s)=(2\pi)^{-2s}\Gamma(s)\Gamma(s-k_2+1)\zeta_{Nm^2}(2s-(k_1+k_2)+2)R_{f_1,g_1}(s)
\end{equation}
where, 
\begin{equation*}
\zeta_{Nm^2}(s)=\prod_{p\vert Nm^2}(1-p^{-s})\zeta(s).
\end{equation*}\\
It is well known that $R^*_{f_1,g_1}(s)$ extends to an entire function on $\mathbb{C}$ (cf. \cite[proposition 4.1.5]{Lei},\cite{Bump} ). Since $\Gamma(z)$ has no zeros for all $z \in \mathbb{C}$, it follows that $\zeta_{Nm^2}(2s-(k_1+k_2)+2)R_{f_1,g_1}(s)$ extends to an entire function.
Now by Rankin-Selberg method one has the following integral representation:
\begin{equation*}
R_{f_1,g_1}(s)= \int_{\Gamma_0(A)\setminus\mathfrak{H}} f_1(z)g_1(z)E^*_{k_1-k_2}(z;s-k_1+1)y^{k_1-2}dxdy, 
\end{equation*}\\
where, $A=Nm^2$, $k$ is a non-negative integer and 
\begin{equation*}
E^*_k(z;s):= \pi^{-s}\Gamma(s+k)E_k(z;s).
\end{equation*} 
For $z\in \mathfrak{H},s\in \mathbb{C}$ and $\sigma\gg 1$;
\begin{equation*}
E_k(z;s):= \zeta(2s+k) \sum_{\gamma=(\begin{array}{c} . \ .\\c \ d \end{array})\in \Gamma_0(A)_\infty \setminus \Gamma_0(A)}\frac{y^s}{(cz+d)^k \vert cz+d \vert^{2s}} \ \  
\end{equation*}
is the non-holomorphic Eisenstein series of weight $k>0$ with level $A$ for the cusp $i\infty$ and 
\begin{equation*}
\Gamma_0(A)_\infty = \bigg \lbrace \pm 1 \bigg ( \begin{array}{c} 1 \ m\\0 \ 1 \end{array} \bigg ) \vert m\in \mathbb{Z} \bigg \rbrace .
\end{equation*}
It is worthwhile to note that for $k>0$ the function $E^*_k(z;s)$ extends to an entire function in $s$ (e.g.\cite{Miy}, Cor. 7.2.11, p. 286).

If possible let us assume that the sequence $\{a(n)b(n)\}_{n\equiv l \pmod m}$ has all but finitely many terms $ \geq 0 $ for $(l,m)=1$. 
If we denote the coefficients of  $\zeta_{Nm^2}(2s-(k_1+k_2)+2)R_{f_1,g_1}(s)$ by $e(n)$ for all $n\in \mathbb{N}$. Then  $e(1)\neq 0$ and by assumption we have that $\{I_l(n)a(n)b(n)\}_{n\geq 1}$ has all but finitely many terms $\geq 0$. 

Now by Lemma \ref{lem4} we have, 
\begin{equation} \label{eq7} \sum_{n\geq 1}e(n)n^{-s} 
\end{equation}
(for $\sigma \gg 1$) must either have a singularity at the real point of its abscissa of convergence or must converge  for all $s\in \mathbb{C}$. The first alternative is excluded as \eqref{eq7} has holomorphic continuation to $\mathbb{C}$.

Now from \eqref{eq6} for $\sigma \gg 1$, we have that 
\begin{equation*}
(2\pi)^{-2s}\Gamma(s)\Gamma(s-k_2+2)\sum_{n\geq 1}e(n)n^{-s}  
\end{equation*} 
extends to an entire function. Since $\Gamma(s)\Gamma(s-k_2+2)$ has poles at $s=0,-1,-2,....$, it follows that $\sum_{n\geq 1} e(n)n^{-s}$ must vanish at these points.  We thus obtain a system of linear equations for $e(n)$ whose determinant is a Vandermonde determinant and thus is  non-zero. It follows then that $e(n)=0$ for all $n\geq 1$, which contradicts the fact that $e(1)\neq 0$. Thus we conclude that the subsequences $(a(n))_{n\equiv l \pmod m}$ and $(b(n))_{n\equiv l\pmod m}$ have infinitely many terms of different signs simultaneously for $(l,m)=1$.
     
Now we assume that the sequence $ \{a(n)b(n)\}_{n\equiv l\pmod m}$ has all but finitely many terms  $\leq 0$ for $(l,m)=1$. Let
\begin{equation*} 
 e(n) = - \ \text{coefficient of} \  \zeta_{Nm^2}(2s-(k_1+k_2)+2)R_{f_1,g_1}(s).
\end{equation*}
Hence $e(1) \neq 0$ and $e(n)\geq 0$ for all but finitely many $n$. Applying Lemma \ref{lem4} we can conclude that $\sum_{n=1}^\infty e(n)n^{-s}$ must converge for all $s \in \mathbb{C}$. Now we can argue previously to conclude $e(n)=0$ for all $n \geq 1$ which contradicts the nonvanishing of $e(1)$. Thus the subsequences $(a(n))_{n\equiv l\pmod m}$ and $(b(n))_{n\equiv l\pmod m}$ have infinitely many terms of same signs simultaneously for $(l,m)=1$. This completes the proof.

\section{Proof of Theorem \ref{thm3}}

Let us denote $\alpha_j= 1/2, 3/4, 7/9$  and $\beta_j= 2/11, 1/9, 2/27$ respectively for $j= 2, 3, 4$. Let
$$
f(z)=\sum_{n\geq 1}a(n)e^{2\pi inz} \in S_{k_1}
$$
and   
$$
g(z)=\sum_{n\geq 1}b(n)e^{2\pi inz} \in S_{k_2}
$$ 
be normalized Hecke eigenform. Here we want to compare the sparse sequences $\{a(n^j)\}_{n\geq1}$ and $\{b(n^j)\}_{n\geq1}$ simultaneously for $j=2, 3, 4$. The following two lemmas will be required to prove the theorem.

\begin{lemma}\label{lem5}
Let $f(z)=\sum\limits_{n\geq 1}a(n)e^{2\pi inz} \in S_k$ be normalized Hecke eigenform. Then for any $\epsilon > 0$ and $j=2, 3, 4$ we have,
$$\sum_{n\leq x}a(n^j)\ll_{f,\epsilon} x^{\alpha_j+\epsilon}$$
where  $\alpha_j$'s are defined above. 
\end{lemma}

\begin{proof}
We refer \cite{Fom} and \cite{Lu} for the proof.
\end{proof}

\begin{lemma}\label{lem6}
Let $f\in S_{k_1}$ and $g \in S_{k_2}$ be normalized Hecke eigenform with Fourier coefficients $a(n)$ and $b(n)$ respectively. Then for any $\epsilon > 0$ and $j= 2, 3, 4$ we have,
$$\sum_{n\leq x}a(n^j)b(n^j)= C_j x+ O_{f,\epsilon}(x^{1-\beta_j+\epsilon})$$
where $C_j$'s are absolute constants and $\beta_j$'s are defined above. 
\end{lemma}

\begin{proof}
The proof of the lemma is immediate consequence of the proof of Theorem $1.1, 1.2, 1.3$ of \cite{Lao}.
\end{proof}

\subsection{Proof of theorem \ref{thm3}}

Let, $h = h(x) = x^{1-\beta_j+2\epsilon}$, where $\epsilon (>0)$ is sufficiently small. If possible let us assume that for $j=2, 3, 4$, the sparse sequence $\{a(n^j)b(n^j)\}_{n\geq1}$ are of constant sign say positive $ \forall ~ n\in (x,x+h]$.

Now from Lemma \ref{lem6} we have,
\begin{align}\label{eq8}
\sum_{x<n\leq x+h} a(n^j)b(n^j)= C_jh + O_{f,\epsilon}(x^{1-\beta_j+\epsilon}) \gg x^{1-\beta_j+2\epsilon}	
\end{align} 

On the other hand using Lemma \ref{lem5} and Delign's bound (cf. \cite{Del} ) we get,
\begin{align}\label{eq9}
\sum_{x<n\leq x+h} a(n^j)b(n^j)\ll x^{2\epsilon}\sum_{x<n\leq x+h}b(n^j) 
&\ll x^{2\epsilon}\bigg((x+h)^{\alpha_j+\epsilon}+x^{\alpha_j+\epsilon}\bigg)& \nonumber\\
&\ll x^{\alpha_j+3\epsilon}&
\end{align}

Now comparing $1-\beta_j$ and $\alpha_j$ in \ref{eq8} and \ref{eq9}, we can see that the bounds contradict each other. Therefore, atleast one $a(n^j)b(n^j)$ for $n\in (x,x+h]$ are of negetive sign. Hence we can conclude that, for $j=2, 3, 4$ the sparse sequences $\{a(n^j)\}_{n\geq1}$ and $\{b(n^j)\}_{n\geq1}$ have infinitely many terms of different signs simultaneously. By similar argument, one can also show that the sparse sequences have infinitely many terms of same signs simultaneously. This completes the proof of the theorem.

\section*{\bf Acknowledgements} The author would like to express hearty thanks to Professor Kalyan Chakraborty for many fruitful discussions. The author is grateful to the anonymous referee for suggesting essential corrections in an earlier version. This research is partly supported by the Infosys scholarship for senior students.

 \end{document}